\documentclass[12pt]{article}
\usepackage{amsmath, amssymb,latexsym}
\usepackage{hyperref}
\usepackage{color}
\definecolor{myorange}{RGB}{180,90,0}
\definecolor{mygreen}{RGB}{70,140,0}
\usepackage[normalem]{ulem}

\usepackage{xcolor}

\definecolor{amethyst}{rgb}{0.6, 0.4, 0.8}

\usepackage[T1]{fontenc}
\usepackage[utf8]{inputenc}
\def\wrtext#1{\relax\ifmmode{\leavevmode\hbox{#1}}\else{#1}\fi}
\def\abs#1{\left|#1\right|}
\def\begeq{\begin{equation}}
\def\endeq{\end{equation}}

\def\part#1{\frac{\partial}{\partial #1}}

\def\norm#1{||\,#1\,||}

\newcommand{\real}{\mbox{\bf R}}
\newcommand{\comp}{\mbox{\bf C}}
\newcommand{\z}{\mbox{\bf Z}}
\newcommand{\nat}{\mbox{\bf N}}

\renewcommand{\exp}{\mbox{\rm exp\,}}

\newtheorem{dref}{Definition}[section]
\newtheorem{lemma}[dref]{Lemma}
\newtheorem{theo}[dref]{Theorem}
\newtheorem{prop}[dref]{Proposition}

\newenvironment{proof}{\vspace{.3cm}\noindent{{\em Proof:}}}{\hfill$\Box$}

\title{Semiclassical Gevrey operators and magnetic translations}
\author{Michael \textsc{Hitrik} \footnote{Department of Mathematics, University of California, Los Angeles CA 90095-1555, USA, {\sf hitrik@math.ucla.edu}} \and Richard \textsc{Lascar} \footnote{JAD - UMR 7351, Universit\'e C\^ote d'Azur Parc Valrose 06108 Nice Cedex 02, France, {\sf richard.lascar@univ-cotedazur.fr}} \and Johannes \textsc{Sj\"ostrand}\footnote{IMB, Universit\'e de Bourgogne 9, Av. A. Savary, BP 47870
FR-21078 Dijon, France and UMR 5584 CNRS, {\sf johannes.sjostrand@u-bourgogne.fr}} \and Maher \textsc{Zerzeri} \footnote{LAGA - UMR7539 CNRS, Universit\'e Sorbonne Paris-Nord, 99, avenue J.-B. Cl\'ement F-93430 Villetaneuse, France, {\sf zerzeri@math.univ-paris13.fr}}}

\date{}

\begin{document}
\maketitle

\begin{flushright}
\textit{In memory of Misha Shubin}
\end{flushright}

\vspace*{1cm}
\noindent
{\bf Abstract}: We study semiclassical Gevrey pseudodifferential operators acting on the Bargmann space of entire functions with quadratic exponential weights. Using so\-me ideas of the time frequency analysis, we show that such operators are uniformly bounded on a natural scale of exponentially weighted spaces of holomorphic functions, provided that the Gevrey index is $\geq 2$.

\tableofcontents
\section{Introduction and statement of results}
\setcounter{equation}{0}
\label{sec_introduction}
The purpose of this paper is to study continuity properties of semiclassical Gevrey pseudodifferential operators acting on exponentially weighted spaces of entire holomorphic functions on $\comp^n$, providing an alternative approach to some of the results established in the recent work~\cite{HiLaSjZe}. Let us proceed to describe the assumptions and state the main results.

\bigskip
\noindent
Let $\Omega \subset \real^m$ be open and let $s\geq 1$. The Gevrey class ${\cal G}^s(\Omega)$ consists of all functions $u\in C^{\infty}(\Omega)$ such that for any $K\subset \Omega$ compact there exist $A>0$, $C>0$ such that for all $\alpha \in \nat^m$ we have
\begeq
\label{intr1}
\abs{\partial^{\alpha} u(x)} \leq A C^{\abs{\alpha}} \left(\alpha!\right)^s,\quad x\in K.
\endeq
The class ${\cal G}^1(\Omega)$ is the space of real analytic functions on $\Omega$, while for $s>1$, we have ${\cal G}^s_0(\Omega) := {\cal G}^s(\Omega) \cap C^{\infty}_0(\Omega) \neq \{0\}$, see~\cite[Theorem 1.3.5]{Hormander_book}. In this work we shall specifically be concerned with the subspace ${\cal G}^s_b(\real^m) \subset {\cal G}^s(\real^m)$ of functions $u\in C^{\infty}(\real^m)$ satisfying the Gevrey condition (\ref{intr1}) uniformly on all of $\real^m$, for some $s>1$: we have $u\in {\cal G}^s_b(\real^m)$ precisely when there exist $A>0$, $C>0$ such that for all $\alpha \in \nat^m$, we have
\begeq
\label{intr1.1}
\abs{\partial^{\alpha} u(x)} \leq A C^{\abs{\alpha}} (\alpha!)^s,\quad x\in \real^m.
\endeq

\bigskip
\noindent
Let $\Phi_0$ be a strictly plurisubharmonic quadratic form on $\comp^n$ and let us introduce the real linear subspace
\begeq
\label{intr2}
\Lambda_{\Phi_0} = \left\{\left(x,\frac{2}{i}\frac{\partial \Phi_0}{\partial x}(x)\right), \, x\in \comp^n\right\} \subset \comp^{2n} = \comp^n_x \times \comp^n_{\xi}.
\endeq
Identifying $\Lambda_{\Phi_0}$ linearly with $\comp^n_x$, via the projection map $\pi_x: \Lambda_{\Phi_0} \ni (x,\xi) \mapsto x\in \comp^n_x$, we may define the Gevrey spaces ${\cal G}^s(\Lambda_{\Phi_0})$, ${\cal G}^s_0(\Lambda_{\Phi_0})$, ${\cal G}^s_b(\Lambda_{\Phi_0})$.

\medskip
\noindent
Given $a\in {\cal G}^s_b(\Lambda_{\Phi_0})$, for some $s>1$, and $u\in {\rm Hol}(\comp^n)$ such that for all $N\geq 0$ we have
$u(x) = {\cal O}_{h,N}(1) \langle{x\rangle}^{-N} e^{\Phi_0(x)/h}$, let us introduce the semiclassical Weyl quantization of $a$ acting on $u$,
\begeq
\label{intr3}
{\rm Op}_h^w(a) u(x) = \frac{1}{(2\pi h)^n}\int\!\!\!\!\int_{\Gamma(x)} e^{\frac{i}{h}(x-y)\cdot \theta} a\left(\frac{x+y}{2},\theta\right)u(y)\, dy\wedge d\theta.
\endeq
Here $0 < h \leq 1$ is the semiclassical parameter and $\Gamma(x)\subset \comp^{2n}_{y,\theta}$ is the natural integration contour given by
\begeq
\label{intr4}
\theta = \frac{2}{i} \frac{\partial \Phi_0}{\partial x}\left(\frac{x+y}{2}\right).
\endeq
The operator ${\rm Op}_h^w(a)$ extends to a uniformly bounded map
\begeq
\label{intr5}
{\rm Op}_h^w(a) = {\cal O}(1): H_{\Phi_0}(\comp^n) \rightarrow H_{\Phi_0}(\comp^n),
\endeq
see~\cite{Sj95},~\cite{HiSj15}. Here $H_{\Phi_0}(\comp^n)$ is the Bargmann space defined by
\begeq
\label{intr6}
H_{\Phi_0}(\comp^n) = {\rm Hol}(\comp^n) \cap L^2(\comp^n, e^{-2\Phi_0/h} L(dx)),
\endeq
with $L(dx)$ being the Lebesgue measure on $\comp^n$. Now the mapping property (\ref{intr5}) follows merely from the fact that $\nabla^k a\in L^{\infty}(\Lambda_{\Phi_0})$ for all $k\in \nat$, and the Gevrey smoothness of $a$ allows us to consider other weights as well. The effect of modifying the exponential weight has been considered in~\cite{HiLaSjZe}, and the following result has been established there, see~\cite[Theorem 3.3, Theorem 3.4]{HiLaSjZe}.

\begin{theo}
\label{theo_main0}
Let $a\in {\cal G}^s_b(\Lambda_{\Phi_0})$, $s>1$, and let $\Phi_1 \in C^{1,1}(\comp^n;\real)$ be such that
\begeq
\label{intr7}
\norm{\nabla^k(\Phi_1 - \Phi_0)}_{L^{\infty}({\bf C}^n)} \leq \frac{1}{C} h^{1 - \frac{1}{s}},\quad k = 0,1,2,
\endeq
where $C>0$ is large enough. Then the operator ${\rm Op}_h^w(a)$ extends to a uniformly bounded map
\begeq
\label{intr8}
{\rm Op}_h^w(a) = {\cal O}(1): H_{\Phi_1}(\comp^n) \rightarrow H_{\Phi_1}(\comp^n).
\endeq
Here, similarly to {\rm (\ref{intr6})}, we have set
$$
H_{\Phi_1}(\comp^n) = {\rm Hol}(\comp^n) \cap L^2(\comp^n, e^{-2\Phi_1/h} L(dx)).
$$
\end{theo}

\bigskip
\noindent
The proof of Theorem \ref{theo_main0} in~\cite{HiLaSjZe} proceeds by a contour deformation argument in (\ref{intr3}), introducing a Gevrey almost holomorphic extension of $a\in {\cal G}^s_b(\Lambda_{\Phi_0})$ and making use of Stokes's theorem. A noteworthy aspect of the proof developed  in~\cite{HiLaSjZe} is that performing a deformation to a contour of the form
\begeq
\label{intr9}
\theta = \frac{2}{i} \frac{\partial \Phi_0}{\partial x}\left(\frac{x+y}{2}\right) + \frac{i}{C}\overline{(x-y)},\quad C>0,
\endeq
natural in the analytic theory~\cite{Sj82},~\cite{Sj95},~\cite{HiSj15}, does not lead to some exponentially accurate remainder estimates of the form
\begeq
\label{intr9.1}
{\cal R} = {\cal O}(1)\, \exp\left(-\frac{1}{{\cal O}(1)} h^{-\frac{1}{s}}\right): H_{\Phi_0}(\comp^n) \rightarrow L^2(\comp^n, e^{-2\Phi_0/h} L(dx)),
\endeq
natural in the Gevrey theory. To overcome this issue, the argument in~\cite{HiLaSjZe} proceeds by deforming to a suitable "mixed" contour, using (\ref{intr9}) in the region
$$
\abs{x-y}\leq \frac{1}{{\cal O}(1)} h^{1 - \frac{1}{s}}
$$
only. An additional deformation to a "mixed" contour adapted to the weight $\Phi_1$ leads then to the uniform boundedness in (\ref{intr8}) in the range $s\in (1,2]$ only, and some further work, involving another change of contour, is required to recover the mapping property (\ref{intr8}) in the full range $s>1$. See~\cite[Theorem 3.4]{HiLaSjZe}. Let us also remark that, as explained in~\cite{HiLaSjZe}, when obtaining a uniformly bounded realization of the operator in (\ref{intr8}) for $s>2$, via a contour deformation, one has to accept a remainder which is larger than the one in (\ref{intr9.1}).

\bigskip
\noindent
Our purpose here is to give a direct proof of Theorem \ref{theo_main0} in the "complementary" region $s\geq 2$, avoiding the use of contour deformations entirely. Specifically, the following is the main result of this work.
\begin{theo}
\label{theo_main}
Let $a\in {\cal G}^s_b(\Lambda_{\Phi_0})$, for some $s\geq 2$, and let $\Phi_1 \in C^{1,1}(\comp^n;\real)$ be such that
\begeq
\label{intr10}
\norm{\nabla^k(\Phi_1 - \Phi_0)}_{L^{\infty}({\bf C}^n)} \leq \frac{1}{C} h^{1 - \frac{1}{s}},\quad k = 0,1,
\endeq
where $C>0$ is large enough. Then the operator ${\rm Op}_h^w(a)$ extends to a uniformly bounded map
\begeq
\label{intr11}
{\rm Op}_h^w(a) = {\cal O}(1): H_{\Phi_1}(\comp^n) \rightarrow H_{\Phi_1}(\comp^n).
\endeq
\end{theo}

\bigskip
\noindent
When establishing Theorem \ref{theo_main}, rather than performing a contour deformation in (\ref{intr3}), we shall proceed by following some basic ideas of the time frequency analysis~\cite{Gr_book},\cite{GrHe},~\cite{Sj08}, decomposing the symbol $a$ into a superposition of coherent states of the form $\Lambda_{\Phi_0} \ni X\mapsto e^{2i\sigma(X,Y)/h} \chi_0((X-T)/h^{1/2})$, for $Y,T\in \Lambda_{\Phi_0}$. Here $\sigma$ is the complex symplectic form on $\comp^{2n}_{x,\xi}$ and $\chi_0$ is a fixed function in ${\cal G}^s_b(\Lambda_{\Phi_0})$, which we can choose essentially as a real Gaussian. Passing to the Weyl quantizations leads to the representation of the operator ${\rm Op}_h^w(a)$ as a direct integral of certain rank one projections expressed in terms of the "magnetic translations" $e^{i\sigma((x,hD_x),Y)/h}$, for $Y\in \Lambda_{\Phi_0}$, unitary on $H_{\Phi_0}(\comp^n)$, whose operator norm on $H_{\Phi_1}(\comp^n)$ can be controlled. Theorem \ref{theo_main} follows from these observations, by an application of Schur's lemma, when combined with the Wiener type characterization of the Gevrey space ${\cal G}^s_b(\Lambda_{\Phi_0})$. We refer to~\cite{Sj94},~\cite{Sj_XEDP_95} for the original works on the Wiener algebras of pseudodifferential operators. See also~\cite{Gr06},~\cite{Sj08}, and the references given there. It is perhaps of note that the time frequency approach to the proof of Theorem \ref{theo_main} appears to work for the Gevrey indices $s\geq 2$ only, which is precisely the range where the contour deformation method of~\cite{HiLaSjZe} encounters some difficulties.

\bigskip
\noindent
The plan of the paper is as follows. In Section \ref{magn_transl}, we study mapping properties of magnetic translations on the weighted spaces
$H_{\Phi_0}(\comp^n)$, $H_{\Phi_1}(\comp^n)$. In Section \ref{Fourier_inv}, as a preparation for the proof of Theorem \ref{theo_main}, we consider  compactly supported Gevrey symbols $a\in {\cal G}^s_0(\Lambda_{\Phi_0})$. In this case, the mapping property (\ref{intr11}) can be established in the full range $s>1$, by decomposing the operator ${\rm Op}^w_h(a)$ into a superposition of magnetic translations directly. Section \ref{Gevrey_Wiener} is devoted to the proof of Theorem \ref{theo_main} in the general case. As alluded to above, an essential role in the proof is played by a decomposition of the operator ${\rm Op}^w_h(a)$ into a direct integral of rank one projections, and we would like to emphasize that it can be viewed as the Bargmann space analogue of the corresponding decomposition established in~\cite{GrHe} in the real setting. In Appendix  \ref{comp_Gevrey} we recall the composition formulas for the Weyl $h$--pseudodifferential calculus in the complex domain, obtained by the method of magnetic translations.

\medskip
\noindent
{\it {We dedicate this paper to the memory of Misha Shubin and would like to acknowledge his pioneering contributions to the global theory of pseudodifferential operators~{\rm \cite{Shubin_book}},~{\rm \cite{TS}}, of which this work is a more recent descendant.}}

\section{Magnetic translations on weighted spaces}
\label{magn_transl}
\setcounter{equation}{0}
Let $\Phi_0$ be a strictly plurisubharmonic quadratic form on $\comp^n$ and let $\Lambda_{\Phi_0}\subset \comp^{2n}$ be defined as in (\ref{intr2}). The real $2n$-dimensional linear subspace $\Lambda_{\Phi_0}$ is I-Lagrangian and R-symplectic, in the sense that the restriction of the complex symplectic (2,0)--form
\begeq
\label{eq2.1.1}
\sigma = \sum_{j=1}^n d\xi_j \wedge dx_j
\endeq
on $\comp^{2n} = \comp^n_x \times \comp^n_{\xi}$ to $\Lambda_{\Phi_0}$ is real and non-degenerate. In particular, $\Lambda_{\Phi_0}$ is maximally totally real.

\bigskip
\noindent
Let $\ell(x,\xi)$ be a complex linear form on $\comp^{2n}$ so that
\begeq
\label{eq2.1.2}
\ell(x,\xi )= \ell'_x\cdot x+\ell'_\xi \cdot \xi = \sigma ((x,\xi),H_\ell),\quad H_\ell = \ell'_\xi\cdot \partial_x-\ell'_x\cdot \partial_\xi.
\endeq
Let us notice that the restriction $\ell|_{\Lambda_{\Phi_0}}$ is real precisely when the Hamilton vector field $H_{\ell}\in T\Lambda_{\Phi_0}$. Here we identify the holomorphic (constant) vector field $H_{\ell}$ with the corresponding real vector field $H_{\ell}^{\rho} = H_{\ell} + \overline{H_{\ell}}$. We have therefore
\begeq
\label{eq2.2}
-\ell'_x = \frac{2}{i}\left((\Phi_0) ''_{xx}\ell'_\xi +(\Phi_0)''_{x\overline{x}}\overline{\ell'_\xi } \right) = \frac{2}{i} \frac{\partial \Phi_0}{\partial x}(\ell'_{\xi}),
\endeq
and we may write
\begeq
\label{eq2.3}
\ell(x,\xi) = \sigma((x,\xi),H_{\ell}) = -\frac{2}{i} \frac{\partial \Phi_0}{\partial x}(x^*)\cdot x + x^*\cdot \xi, \quad (x,\xi) \in \comp^{2n},
\endeq
for some unique $\comp^n \ni x^* = \ell'_{\xi}$.

\bigskip
\noindent
To a complex linear form $\ell$ on $\comp^{2n}$, such that $\ell|_{\Lambda_{\Phi_0}}$ is real, we associate the operator
\begeq
\label{eq2.3.1}
{\rm Op}_h^w\left(e^{-i\ell/h}\right) = e^{-i\ell(x,hD_x)/h},
\endeq
and recall from~\cite[Proposition 1.4]{Sj95} that
\begeq
\label{eq2.4}
e^{-i\ell(x,hD_x)/h} = e^{-\frac{i}{2h} \ell'_x\cdot x} \circ \tau_{\ell'_{\xi}} \circ e^{-\frac{i}{2h} \ell'_x\cdot x}.
\endeq
Here $\tau_z$ is the operator of translation by $z\in \comp^n$, $(\tau_z u)(x) = u(x-z)$. Throughout this paper, operators of the form (\ref{eq2.4}) will be referred to as magnetic translations, in view of the fact that such operators appear naturally in the study of Schr\"odinger operators with magnetic fields, see~\cite{Sj91},~\cite{DGR}.

\bigskip
\noindent
Let us recall from~\cite{Sj95} that the first order differential operator $\ell(x,hD_x)$ is selfadjoint on the Bargmann space $H_{\Phi_0}(\comp^n)$, when equipped with the maximal domain $\{u\in H_{\Phi_0}(\comp^n);\, \ell(x,hD_x) u\in H_{\Phi_0}(\comp^n)\}$. In particular, the operator in (\ref{eq2.3.1}) is unitary when acting on $H_{\Phi_0}(\comp^n)$. For future reference, it will be convenient for us to start by doing the exercise mentioned in the proof of~\cite[Proposition 1.4]{Sj95}, verifying the unitarity directly, using the expression (\ref{eq2.4}). See also~\cite[Section 3]{Sj08}.

\bigskip
\noindent
Given $u\in H_{\Phi_0}(\comp^n)$, we shall show that
\begeq
\label{eq2.5}
\abs{(e^{-\Phi _0/h}e^{-i\ell(x,hD_x)/h}u)(x)} = \abs{(e^{-\Phi _0/h}u)(x-\ell'_\xi)},\quad x\in \comp^n.
\endeq
When doing so, let us write using (\ref{eq2.4}),
\begeq
\label{eq2.6}
e^{-i\ell(x,hD_x)/h}u(x) = e^{\frac{i}{2h} \ell'_x\cdot \ell'_{\xi}} e^{-\frac{i}{h} \ell'_x\cdot x}u(x-\ell'_{\xi}).
\endeq
A simpler expression is obtained, by exploiting the natural symmetry in the Weyl quantization, if we express $e^{-i\ell(x,hD_x)/h}u$ at the point
$x+\ell '_\xi /2$ in terms of $u$ at the point $x-\ell '_\xi /2$, which gives that
\begeq
\label{eq2.6.1}
(e^{-i\ell(x,hD_x)/h}u)\left(x+\frac{\ell'_\xi}{2}\right) = e^{-i\ell '_x\cdot x/h}u\left(x-\frac{\ell'_\xi}{2}\right).
\endeq
It follows that
\begin{multline}
\label{eq2.7}
\left(e^{-\frac{1}{h}{\Phi _0}}e^{-i\ell(x,hD_x)/h}u\right)\left(x+\frac{\ell'_\xi}{2}\right)
\\ = e^{-\frac{1}{h}\left( {\Phi _0} (x+\frac{1}{2}\ell '_\xi )-{\Phi _0}
  (x-\frac{1}{2}\ell '_\xi ) \right)}e^{-i\ell '_x\cdot
x/h}\left(e^{-\frac{1}{h}{\Phi _0} } u\right)\left(x-\frac{\ell'_\xi}{2}\right).
\end{multline}
We next observe that
\begeq
\label{eq2.8}
\Phi_0\left(x+\frac{1}{2}\ell '_\xi\right)- \Phi_0\left(x-\frac{1}{2}\ell '_\xi\right) = 2{\rm Re}\, \left(\partial _x{\Phi _0} (x)\cdot \ell '_\xi\right),
\endeq
in view of the following standard consequence of Taylor's formula
\begeq
\label{eq2.9}
\Phi_0(x) - \Phi_0(y) = 2 {\rm Re}\, \biggl(\partial_x \Phi_0\left(\frac{x+y}{2}\right) \cdot (x-y) \biggr),
\endeq
valid for the real valued quadratic form $\Phi_0$. Moreover, using (\ref{eq2.2}) we see that the right hand side of (\ref{eq2.8}) takes the form
\begin{multline}
\label{eq2.10}
2 {\rm Re}\, \left(\partial_x{\Phi _0} (x)\cdot \ell '_\xi\right) = 2 {\rm Re}\, \left(\partial_x{\Phi _0} (\ell'_{\xi})\cdot x\right) \\
= - {\rm Re}\, \left(i\left(-\frac{2}{i}\partial_x {\Phi _0} (\ell '_\xi )\cdot x \right)\right) = - {\rm Re}\, \left(i\ell '_x\cdot x \right).
\end{multline}
Here we have also used the general relation
\begeq
\label{eq2.11}
{\rm Re}\, \biggl(\partial_x \Phi_0(x)\cdot y \biggr) = {\rm Re}\, \biggl(\partial_x \Phi_0(y)\cdot x\biggr),\quad x,y\in \comp^n.
\endeq
Combining (\ref{eq2.8}) and (\ref{eq2.10}), we obtain that
\begeq
\label{eq2.11.1}
{\Phi _0}\left(x+\frac{1}{2}\ell '_\xi\right) - {\Phi _0}\left(x-\frac{1}{2}\ell '_\xi\right) + {\rm Re}\, \left(i\ell '_x\cdot x \right)=0,
\endeq
and the absolute value of the prefactor in the right hand side of (\ref{eq2.7}) is therefore equal to $1$. We get from (\ref{eq2.7}) and (\ref{eq2.11.1}),
$$
\left|\left(e^{-\frac{1}{h}{\Phi_0}}e^{-i\ell(x,hD_x)/h}u\right)\left(x+\frac{\ell'_\xi}{2}\right)\right|
= \left| \left(e^{-\frac{1}{h}{\Phi _0} } \right) u\left(x-\frac{\ell'_\xi}{2}\right)\right|,
$$
and (\ref{eq2.5}) follows, by replacing $x$ by $x-\ell '_\xi /2$.

\bigskip
\noindent
It follows from (\ref{eq2.5}) that the operator
\begeq
\label{eq2.12}
e^{-i\ell(x,hD_x)/h}: H_{\Phi_0}(\comp^n) \rightarrow H_{\Phi_0}(\comp^n)
\endeq
is an isometry, and is therefore unitary, since it is a bijection, with the inverse given by $e^{i\ell(x,hD_x)/h}: H_{\Phi_0}(\comp^n) \rightarrow H_{\Phi_0}(\comp^n)$.

\bigskip
\noindent
{\it Remark}. Magnetic translations play a role also in the theory of Toeplitz operators, where they appear under the name of the Weyl operators, see~\cite{LC}.

\bigskip
\noindent
We shall now consider weights other than $\Phi_0$. To this end, let $\Phi_1 \in C^{1,1}(\comp^n;\real)$, the space of $C^1$--functions on $\comp^n$ with a globally Lipschitz gradient. In particular we know, thanks to Rademacher's theorem, that
\begeq
\label{eq2.13}
\nabla^2 \Phi_1 \in L^{\infty}(\comp^n).
\endeq
Let us also assume that $\norm{\nabla^k(\Phi_1 - \Phi_0)}_{L^{\infty}({\bf C}^n)}$ are small enough, for $k=0,1$. We would like to consider magnetic translations acting on the weighted space
\begeq
\label{eq2.14}
H_{\Phi_1}(\comp^n) = L^2(\comp^n, e^{-2\Phi_1/h} L(dx)) \cap {\rm Hol}(\comp^n).
\endeq
To this end, viewing the operator in (\ref{eq2.4}) as a Fourier integral operator associated to the complex canonical transformation $\exp(H_{\ell})$,
we shall first determine the image of the Lipschitz manifold
\begeq
\label{eq2.15}
\Lambda_{\Phi_1} = \left\{\left(x,\frac{2}{i}\frac{\partial \Phi_1}{\partial x}(x)\right), \, x\in \comp^n\right\} \subset \comp^{2n},
\endeq
under the map $\exp(H_{\ell}): \comp^{2n} \ni \rho \mapsto \rho + H_{\ell}\in \comp^{2n}$. Here we may notice that the manifold $\Lambda_{\Phi_1}$ is I-Lagrangian, in the sense that its almost everywhere defined tangent plane is Lagrangian with respect to ${\rm Im}\, \sigma$.

\bigskip
\noindent
We have the following result, where we write $\Phi_1(x) = \Phi_0(x) + f(x)$.
\begin{prop}
\label{prop_weights}
Let $\ell(x,\xi)$ be a complex linear form on $\comp^{2n}$ such that $\ell|_{\Lambda_{\Phi_0}}$ is real, and let us represent $\ell$ in the form {\rm (\ref{eq2.3})}. Then we have
$$
\exp(H_{\ell})\left(\Lambda_{\Phi_1}\right) = \Lambda_{\Phi_2},
$$
where $\Phi_2\in C^{1,1}(\comp^n)$ is given by
\begeq
\label{eq2.16}
\Phi_2(x) = \Phi_1(x) + f(x-x^*) - f(x),\quad x \in \comp^n.
\endeq
\end{prop}
\begin{proof}
Following the proof of~\cite[Lemma 2.2]{CoHiSjWh}, let us consider the real Hamilton-Jacobi equation
\begeq
\label{eq2.17}
\frac{\partial \Psi}{\partial t}(x,t) - {\rm Im}\, \ell\left(x,\frac{2}{i}\frac{\partial \Psi}{\partial x}(x,t)\right) = 0, \quad \Psi(x,0) = \Phi_1(x),
\endeq
for $x\in \comp^n$, $t\in \real$, $t\geq 0$. Associated to the function $\Psi(x,t)\in C^{1,1}(\comp^n\times \real;\real)$ in (\ref{eq2.17}) is the Lipschitz manifold
\begeq
\label{eq2.18}
L_{\Psi} = \left\{\left(t, \frac{\partial \Psi}{\partial t}(x,t), x, \frac{2}{i}\frac{\partial \Psi}{\partial x}(x,t)\right)\right\} \subset \real^2_{t,\tau} \times \comp^{2n}_{x,\xi},
\endeq
which is Lagrangian with respect to the real symplectic form
\begeq
\label{eq2.19}
d\tau \wedge dt - {\rm Im}\, \sigma.
\endeq
The function $\tau - {\rm Im}\, \ell$ vanishes along $L_{\Psi}$, in view of (\ref{eq2.17}), and therefore its Hamilton vector field, computed with respect to the real symplectic form (\ref{eq2.19}), is tangent to $L_{\Psi}$, almost everywhere. Using the general relation
$$
H_{\ell}^{\rho} = H_{-{\rm Im}\, \ell}^{-{\rm Im}\, \sigma},
$$
valid for any $\ell(x,\xi)$ holomorphic, where $H_{\ell}^{\rho} = H_{\ell} + \overline{H_{\ell}}$ is the real vector field naturally associated to the holomorphic vector field $H_{\ell}$, see~\cite{Sj82}, we conclude that the vector field
$$
\partial_t + H_{-{\rm Im}\, \ell}^{-{\rm Im}\, \sigma} = \partial_t + H_{\ell}^{\rho}
$$
is tangent to $L_{\Psi}$. Identifying $H_{\ell}^{\rho}$ and $H_{\ell}$, we get therefore
$$
\Lambda_{\Psi(\cdot,t)} = \exp(t H_{\ell})\left(\Lambda_{\Phi_1}\right),\quad t\in \real.
$$
We claim now that the unique $C^{1,1}$--solution of (\ref{eq2.17}) is given by
\begeq
\label{eq2.20}
\Psi(x,t) = \Phi_1(x) + f(x-tx^*) - f(x) = \Phi_0(x) + f(x-tx^*).
\endeq
Indeed, using (\ref{eq2.20}), (\ref{eq2.3}), and the fact that $\ell|_{\Lambda_{\Phi_0}}$ is real we see that
\begeq
\label{eq2.21}
{\rm Im}\, \ell\left(x,\frac{2}{i}\frac{\partial \Psi}{\partial x}(x,t)\right) = {\rm Im}\left(x^* \cdot \frac{2}{i} \frac{\partial f}{\partial x} (x-tx^*)\right) = -2 {\rm Re}\left(x^* \cdot \frac{\partial f}{\partial x} (x-tx^*)\right),
\endeq
which agrees with
$$
\frac{\partial \Psi}{\partial t}(x,t) = \partial_t \left(f(x-t x^*)\right) = -\biggl(x^* \cdot f'_x(x-tx^*) + \overline{x^*}\cdot f'_{\overline{x}}(x-tx^*)\biggr).
$$
This shows (\ref{eq2.20}) and completes the proof.
\end{proof}

\bigskip
\noindent
Proposition \ref{prop_weights} can be viewed as an indication that we have a uniformly bounded Fourier integral operator,
\begeq
\label{eq2.22}
e^{-i\ell(x,hD_x)/h} = {\cal O}(1): H_{\Phi_1}(\comp^n) \rightarrow H_{\Phi_2}(\comp^n).
\endeq
Here the weighted space of holomorphic functions $H_{\Phi_2}(\comp^n)$ is defined analogously to (\ref{eq2.14}). In order to give a direct verification of the mapping property (\ref{eq2.22}), we observe that (\ref{eq2.5}) gives, for $u\in H_{\Phi_1}(\comp^n)$,
\begeq
\label{eq2.23}
\abs{(e^{-\Phi_2/h}e^{-i\ell(x,hD_x)/h}u)(x)} = \abs{(e^{-\Phi_1/h}u)(x-x^*)},\quad x\in \comp^n.
\endeq
Here $\Phi_2$ is given by (\ref{eq2.16}). It follows from (\ref{eq2.23}) that the operator in (\ref{eq2.22}) is an isometry, and therefore unitary.

\bigskip
\noindent
The discussion in this section may be summarized in the following result.
\begin{theo}
\label{theo_weights}
Let $\ell(x,\xi)$ be a complex linear form on $\comp^{2n}$ such that $\ell|_{\Lambda_{\Phi_0}}$ is real, and let us represent $\ell$ in the form {\rm (\ref{eq2.3})}. We have the unitary operators
\begeq
\label{eq2.24}
e^{-i\ell(x,hD_x)/h}: H_{\Phi_0}(\comp^n) \rightarrow H_{\Phi_0}(\comp^n),
\endeq
and
\begeq
\label{eq2.25}
e^{-i\ell(x,hD_x)/h}: H_{\Phi_1}(\comp^n) \rightarrow H_{\Phi_2}(\comp^n).
\endeq
Here in {\rm (\ref{eq2.25})}, the weight function $\Phi_1 = \Phi_0 + f \in C^{1,1}(\comp^n;\real)$ is such that $\norm{\nabla^k(\Phi_1 - \Phi_0)}_{L^{\infty}({\bf C}^n)}$ are small enough, for $k=0,1$, and $\Phi_2(x) = \Phi_0(x) + f(x-x^*)$.
\end{theo}

\medskip
\noindent
{\it Remark}. The purpose of this remark is to outline an alternative approach to the Hamilton-Jacobi equation (\ref{eq2.17}), leading directly the unitarity of the operator in (\ref{eq2.25}). To this end, let $u\in H_{\Phi_1}(\comp^n)$, and let us differentiate formally the scalar product
\begeq
\label{eq2.26}
(e^{-it\ell(x,hD_x)/h}u, e^{-it\ell(x,hD_x)/h}u)_{H_{\Psi_t}} = (u(t),u(t))_{H_{\Psi_t}}
\endeq
with respect to $t\in \real$. Here $\Psi_t\in C^{1,1}$ is to be chosen so that the time derivative of (\ref{eq2.26}) vanishes. We refer to~\cite{HeSjSt},~\cite{Sj10},~\cite{HiPrVi},~\cite{HiSj} for other instances of this approach, which is particularly straightforward in the present linear setting. We get
\begin{multline}
\label{eq2.27}
h\partial_t (u(t),u(t))_{H_{\Psi_t}} = h\partial_t \int u(t,x) \overline{u(t,x)} e^{-2\Psi_t(x)/h}\, L(dx) \\
= -i(\ell(x,hD_x)u(t),u(t))_{H_{\Psi_t}} + i (u(t),\ell(x,hD_x)u(t))_{H_{\Psi_t}} \\
- \int 2 \partial_t \Psi_t(x) \abs{u(t,x)}^2 e^{-2\Psi_t(x)/h}\, L(dx).
\end{multline}
Using that $hD_x(\overline{u(t,x)}) = 0$, we obtain that
\begin{multline}
\label{eq2.28}
(\ell(x,hD_x)u(t),u(t))_{H_{\Psi_t}} = \int \ell(x,hD_x)u(t,x)\, \overline{u(t,x)} e^{-2\Psi_t(x)/h}\, L(dx) \\
= \int \abs{u(t,x)}^2 \ell^{{\rm t}}(x,hD_x) \left(e^{-2\Psi_t(x)/h}\right)\, L(dx).
\end{multline}
Here
$$
\ell^{{\rm t}}(x,hD_x) = \ell(x,-hD_x)
$$
is the transpose of the first order differential operator $\ell(x,hD_x)$, and therefore we get the quantization-multiplication formula in its exact version, see also~\cite{HiSj15},
\begeq
\label{eq2.29}
(\ell(x,hD_x)u(t),u(t))_{H_{\Psi_t}} = \int \ell\left(x,\frac{2}{i}\frac{\partial \Psi_t}{\partial x}(x)\right) \abs{u(t,x)}^2 e^{-2\Psi_t(x)/h}\, L(dx).
\endeq
Combining (\ref{eq2.27}) and (\ref{eq2.29}), we get
\begeq
\label{eq2.30}
h\partial_t (u(t),u(t))_{H_{\Psi_t}} = -2 \int \left(\partial_t \Psi_t(x) - {\rm Im}\, \ell\left(x,\frac{2}{i}\frac{\partial \Psi_t}{\partial x}(x)\right)\right) \abs{u(t,x)}^2 e^{-2\Psi_t(x)/h}\, L(dx).
\endeq
The unitarity of the map
$$
e^{-it\ell(x,hD_x)/h}: H_{\Phi_1}(\comp^n) \rightarrow H_{\Psi_t}(\comp^n),\quad t \in \real,
$$
is therefore implied by the vanishing of the first factor under the integral sign in the right hand side of (\ref{eq2.30}), which agrees with the Hamilton-Jacobi equation (\ref{eq2.17}).

\section{Fourier inversion and compactly supported Gevrey symbols}
\label{Fourier_inv}
\setcounter{equation}{0}
The purpose of this section is to prove Theorem \ref{theo_main} in the special case when $a\in {\cal G}^s_0(\Lambda_{\Phi_0})$ is compactly supported. Here we shall let $s>1$ be arbitrary. When doing so, let us start by recalling the invariant form of the Fourier inversion formula on a real symplectic vector space $(W,\sigma)$ of dimension $2n$, see~\cite{T02}. Let
\begeq
\label{fourier1}
{\cal F}u(X) = \widehat{u}(X) = \frac{1}{\pi^n} \int e^{2i\sigma(X,Y)} u(Y)\, dY,\quad u\in {\cal S}(W),
\endeq
be the (twisted) Fourier transformation on $W$. Here $dY$ is the symplectic volume form on $W$. Then ${\cal F}^2 = I$ on ${\cal S}(W)$ and the Fourier transformation extends to a unitary selfadjoint involution of $L^2(W)$. After a change of variables, we get the following semiclassical version of the Fourier inversion formula,
\begeq
\label{fourier2}
a(X) = \frac{1}{\pi^n h^{2n}} \int e^{2i\sigma(X,Y)/h} \widehat{a}\left(\frac{Y}{h}\right)\,dY, \quad a\in {\cal S}(W).
\endeq
Specializing (\ref{fourier2}) to the case when $W = \Lambda_{\Phi_0}$, with the symplectic form given by $\sigma|_{\Lambda_{\Phi_0}}$, where $\sigma$ is the complex symplectic $(2,0)$--form on $\comp^{2n}$ defined in (\ref{eq2.1.1}), $a\in {\cal G}^s_0(\Lambda_{\Phi_0})$, and passing to the Weyl quantizations, we get
\begeq
\label{fourier3}
a^w(x,hD_x):= {\rm Op}_h^w(a) = \frac{1}{\pi^n h^{2n}} \int_{\Lambda_{\Phi_0}} \widehat{a}\left(\frac{Y}{h}\right) e^{2i\sigma((x,hD_x),Y)/h}\,dY.
\endeq
Here for $Y\in \Lambda_{\Phi_0}$, the complex linear form
\begeq
\label{fourier4}
\ell_Y(x,\xi) = \sigma((x,\xi), Y),\quad (x,\xi)\in \comp^{2n}
\endeq
is real along $\Lambda_{\Phi_0}$ and Theorem \ref{theo_weights} provides us therefore with some precise mapping properties for the magnetic translations $e^{2i\sigma((x,hD_x),Y)/h}$, for $Y\in \Lambda_{\Phi_0}$.

\bigskip
\noindent
It is now easy to finish the proof of Theorem \ref{theo_main} in the special case of compactly supported Gevrey symbols. Let $\Phi_1 = \Phi_0 + f \in C^{1,1}(\comp^n;\real)$ be such that
\begeq
\label{fourier5}
\norm{\nabla^k f}_{L^{\infty}({\bf C}^n)} \leq \frac{1}{C} h^{1 - \frac{1}{s}},\quad k =0,1,
\endeq
for some $C>0$. For each $Y = (y,\eta)\in \Lambda_{\Phi_0}$, in view of Theorem \ref{theo_weights}, we have the unitary operators
\begeq
\label{fourier6}
e^{2i\sigma((x,hD_x),Y)/h}: H_{\Phi_1}(\comp^n) \rightarrow H_{\Phi_2}(\comp^n),
\endeq
where $\Phi_2(x) = \Phi_0(x) + f(x+2y)$. It follows that
\begin{multline}
\label{fourier7}
\norm{e^{2i\sigma((x,hD_x),Y)/h}}_{{\cal L}(H_{\Phi_1}({\bf C}^n), H_{\Phi_1}({\bf C}^n))} \\
\leq \exp\left(\frac{\norm{\Phi_2 - \Phi_1}_{L^{\infty}({\bf C}^n)}}{h}\right) = \exp\left(\frac{\norm{f(\cdot + 2y) - f(\cdot)}_{L^{\infty}({\bf C}^n)}}{h}\right).
\end{multline}
Here
\begeq
\label{fourier8}
\norm{f(\cdot + 2y) - f(\cdot)}_{L^{\infty}({\bf C}^n)} \leq {\rm min}\left(2\norm{f}_{L^{\infty}({\bf C}^n)}, 2\norm{\nabla f}_{L^{\infty}({\bf C}^n)}\abs{y}\right),
\endeq
and combining (\ref{fourier5}), (\ref{fourier7}), and (\ref{fourier8}), we get with the same constant $C>0$ as in (\ref{fourier5}),
\begeq
\label{fourier9}
\norm{e^{2i\sigma((x,hD_x),Y)/h}}_{{\cal L}(H_{\Phi_1}({\bf C}^n), H_{\Phi_1}({\bf C}^n))} \leq \exp\left(\frac{2}{C} h^{-1/s} {\rm min}(1,\abs{y})\right),\quad Y = (y,\eta) \in \Lambda_{\Phi_0}.
\endeq
We claim that the integral in the right hand side of (\ref{fourier3}) converges in the space ${\cal L}(H_{\Phi_1}(\comp^n),H_{\Phi_1}(\comp^n))$, provided that $C>0$ in (\ref{fourier5}) is large enough. To this end let us recall from~\cite[Section 2]{HiLaSjZe} the following decay estimate for the (twisted) semiclassical Fourier transform of $a\in {\cal G}^s_0(\Lambda_{\Phi_0})$, see also~\cite[Lemma 12.7.4]{Hormander_book},
\begeq
\label{fourier10}
\abs{\widehat{a}\left(\frac{Y}{h}\right)} \leq C_0\, \exp \left(-\frac{1}{C_0} \left(\frac{\abs{y}}{h}\right)^{1/s}\right),\quad Y = (y,\eta) \in \Lambda_{\Phi_0},\quad C_0 > 0.
\endeq
We get, using (\ref{fourier3}), (\ref{fourier9}), and (\ref{fourier10}), for some $\widetilde{C}>0$,
\begin{multline}
\label{fourier11}
\norm{{\rm Op}_h^w(a)}_{{\cal L}(H_{\Phi_1}({\bf C}^n), H_{\Phi_1}({\bf C}^n))} \\ \leq \frac{1}{\pi^n h^{2n}} \int_{\Lambda_{\Phi_0}} \abs{\widehat{a}\left(\frac{Y}{h}\right)} \norm{e^{2i\sigma((x,hD_x),Y)/h}}_{{\cal L}(H_{\Phi_1}({\bf C}^n), H_{\Phi_1}({\bf C}^n))}\,dY \\
\leq \frac{\widetilde{C}}{h^{2n}} \int_{{\bf C}^n} \exp\left(\frac{1}{h^{1/s}} \left(-\frac{1}{C_0} \abs{y}^{1/s} + \frac{2}{C} {\rm min}(1,\abs{y})\right)\right)\, L(dy).
\end{multline}
Here $L(dy)$ is the Lebesgue measure on $\comp^n$. When estimating the integral in the right hand side of (\ref{fourier11}), we write, using that ${\rm min}(1,\abs{y})\leq \abs{y}^{1/s}$ for $y\in \comp^n$, and taking $C>0$ large enough,
\begin{multline}
\label{fourier12}
\frac{1}{h^{2n}} \int_{{\bf C}^n} \exp\left(\frac{1}{h^{1/s}} \left(-\frac{1}{C_0} \abs{y}^{1/s} + \frac{2}{C}{\rm min}(1,\abs{y})\right)\right)\, L(dy) \\
\leq \frac{1}{h^{2n}} \int_{{\bf C}^n} \exp\left(\frac{1}{h^{1/s}} \left(-\frac{1}{C_0} \abs{y}^{1/s} + \frac{2}{C}\abs{y}^{1/s}\right)\right)\, L(dy) \\
= \frac{1}{h^{2n}} \int_{{\bf C}^n} \exp\left(-\frac{\abs{y}^{1/s}}{h^{1/s}} \left(\frac{1}{C_0} - \frac{2}{C}\right)\right)\, L(dy) \\
= \int_{{\bf C}^n} \exp\left(-\abs{z}^{1/s}\left(\frac{1}{C_0} - \frac{2}{C}\right)\right)\, L(dz) = {\cal O}(1).
\end{multline}
Combining (\ref{fourier11}) and (\ref{fourier12}), we get
\begeq
\label{fourier13}
\norm{{\rm Op}_h^w(a)}_{{\cal L}(H_{\Phi_1}({\bf C}^n), H_{\Phi_1}({\bf C}^n))} \leq {\cal O}(1).
\endeq
This completes the proof of Theorem \ref{theo_main}, in the full range $s>1$, in the case when $a\in {\cal G}^s_0(\Lambda_{\Phi_0})$.

\section{Wiener conditions and rank one decompositions}
\label{Gevrey_Wiener}
\setcounter{equation}{0}
In the beginning of this section, rather than working on $\Lambda_{\Phi_0}$, for simplicity we shall work on $\real^m\simeq T^*\real^n$, where $m=2n$. Following~\cite{Sj94},~\cite{Sj_XEDP_95}, we shall first establish a Wiener type characterization of the Gevrey space ${\cal G}^s_b(\real^m)$, for $s>1$. See also~\cite{T19}. When doing so, let $e_1,\ldots,e_m$ be a basis of $\real^m$ and let
\begeq
\label{Wien1}
\Gamma = \bigoplus_{j=1}^m \z e_j
\endeq
be the corresponding integer lattice. Let $0\leq \chi_0 \in C^{\infty}_0(\real^m)$ be such that
\begeq
\label{Wien2}
\sum_{j \in \Gamma} \chi_{j} = 1, \quad \chi_{j}(x) = \chi_0(x-j).
\endeq
It was remarked in~\cite{Sj94} that we have $a\in S^0_{0,0}(\real^m)$, the space of $C^{\infty}$ functions on $\real^m$ bounded together with all of their derivatives, precisely when $a\in {\cal S}'(\real^m)$ is such that
\begeq
\label{Wien3}
\underset{j \in \Gamma} {\rm sup} \abs{{\cal F}(\chi_{j} a)(\xi)} \leq {\cal O}_N(1) \langle{\xi\rangle}^{-N},\quad N=1,2,\ldots, \quad \xi \in \real^m.
\endeq
Here ${\cal F}$ stands for the standard Fourier transformation, ${\cal F}u(\xi) = \widehat{u}(\xi) = \int e^{-ix\cdot \xi} u(x)\, dx$. When establishing an analogous characterization of the space ${\cal G}^s_b(\real^m) \subset S^0_{0,0}(\real^m)$, it will be natural to assume that the compactly supported function $\chi_0$ in (\ref{Wien2}) satisfies $\chi_0 \in {\cal G}_0^s(\real^m)$.

\begin{prop}
\label{Wiener-Gevrey}
Let $0\leq \chi_0 \in {\cal G}_0^s(\real^m)$ be such that {\rm (\ref{Wien2})} holds. We have $a\in {\cal G}^s_b(\real^m)$, for some $s>1$, if and only if $a\in {\cal S}'(\real^m)$ has the property that
\begeq
\label{Wien4}
\underset{j \in \Gamma} {\rm sup} \abs{{\cal F}(\chi_{j} a)(\xi)} \leq {\cal O}(1)\, \exp\left(-\frac{1}{C}\abs{\xi}^{1/s}\right), \quad \xi \in \real^m,
\endeq
for some $C>0$.
\end{prop}
\begin{proof}
Let us first verify the necessity of (\ref{Wien4}). When doing so, we remark that given $\chi_0 \in {\cal G}^s_0(\real^m)$, $a\in {\cal G}^s_b(\real^m)$, we have, in view of the Leibniz formula,
\begeq
\label{Wien5}
\abs{\partial^{\alpha} \left(\chi_j(x) a(x)\right)} \leq C^{\abs{\alpha}+1} (\alpha!)^s, \quad x\in \real^m, \quad \alpha \in \nat^m,
\endeq
for some $C>0$, uniformly in $j\in \Gamma$. Following an argument in~\cite[Section 2]{HiLaSjZe} and writing
\begeq
\label{Wien6}
{\cal F}\left((1-\Delta)^{\frac{N}{2}}(\chi_j a)\right)(\xi) = \langle{\xi\rangle}^N {\cal F}(\chi_j a)(\xi),
\endeq
for some even integer $N$ large to be chosen, we get in view of (\ref{Wien5}), (\ref{Wien6}),
\begeq
\label{Wien7}
\abs{{\cal F}(\chi_j a)(\xi)} \leq \langle{\xi\rangle}^{-N} \norm{(1-\Delta)^{N/2} (\chi_j a)}_{L^1({\bf R}^m)} \leq C^{N+1} \langle{\xi\rangle}^{-N} (N!)^s,
\endeq
uniformly in $j\in \Gamma$. Choosing $N \sim (\abs{\xi}/C)^{1/s}$, as explained in~\cite{HiLaSjZe}, we get (\ref{Wien4}).

\medskip
\noindent
As for the sufficiency of (\ref{Wien4}), let $a\in {\cal S}'(\real^m)$ be such that (\ref{Wien4}) holds. As remarked above, we then have $a\in S^0_{0,0}(\real^m)$, and we only need to control the growth of the derivatives of $a$. When doing so, let us set $U_j(\xi) = (1/(2\pi)^m){\cal F}(\chi_{j} a)(\xi)$. Let $0\leq \widetilde{\chi}_0 \in {\cal G}^s_0(\real^m)$ be such that $\widetilde{\chi}_0 = 1$ near ${\rm supp}\, \chi_0$ and let us put $\widetilde{\chi}_j(x) = \widetilde{\chi}_0(x-j)$, $j\in \Gamma$. Using the Fourier inversion formula and (\ref{Wien2}), we may write
\begeq
\label{Wien8}
a(x) = \sum_{j\in \Gamma} \widetilde{\chi}_j(x) \int e^{ix\cdot \xi} U_j(\xi)\, d\xi,
\endeq
and therefore, when $\alpha \in \nat^m$, we have
\begeq
\label{Wien9}
D^{\alpha}a(x) = \sum_{j\in \Gamma} \sum_{\beta \leq \alpha} {\alpha \choose \beta} \left(D^{\alpha - \beta} \widetilde{\chi}_0\right)(x-j) \int e^{ix\cdot \xi} \xi^{\beta} U_j(\xi)\, d\xi.
\endeq
Using (\ref{Wien4}), we get, passing to polar coordinates and making a change of variables,
\begeq
\label{Wien10}
\abs{\int e^{ix\cdot \xi} \xi^{\beta} U_j(\xi)\, d\xi} \leq {\cal O}(1) \int \abs{\xi}^{\abs{\beta}} \exp\left(-\frac{1}{C}\abs{\xi}^{1/s}\right)\, d\xi
\leq C_1^{1+\abs{\beta}} \Gamma(s(\abs{\beta} + m)),
\endeq
for some $C_1 > 0$, uniformly in $j$. Here $\Gamma(x) = \int_0^{\infty} e^{-t} t^{x-1}\,dt$, $x>0$, is the $\Gamma$-function. Using Stirling's formula
\begeq
\label{Wien11}
\Gamma(\lambda) \sim \sqrt{\frac{2\pi}{\lambda}} \left(\frac{\lambda}{e}\right)^{\lambda}, \quad \lambda \rightarrow \infty,
\endeq
and the general inequality $\abs{\beta}! \leq m^{\abs{\beta}} \beta!$, valid for $\beta \in \nat^m$, we see therefore that
\begeq
\label{Wien12}
\abs{\int e^{ix\cdot \xi} \xi^{\beta} U_j(\xi)\, d\xi} \leq C^{1+\abs{\beta}} (\beta!)^s,
\endeq
for a (new) constant $C>0$, uniformly in $j$. Combining (\ref{Wien9}) and (\ref{Wien12}), we obtain that
\begeq
\label{Wien13}
\abs{\partial^{\alpha} a(x)} \leq C^{1+\abs{\alpha}} (\alpha!)^s, \quad x\in \real^m.
\endeq
The proof is complete.
\end{proof}

\bigskip
\noindent
{\it Remark}. The proof of the sufficiency of the condition (\ref{Wien4}) can alternatively be carried out by staying on the Fourier transform side. Indeed, let $\chi \in {\cal G}^s_0(\real^m)$, and let us write using (\ref{Wien2}),
\begeq
\label{Wien13.1}
\chi a = \sum_{j;\, {\rm supp}\, \chi\, \cap\, {\rm supp}\, \chi_j \neq \emptyset} \chi \chi_j a,
\endeq
\begeq
\label{Wien13.2}
{\cal F}(\chi a) = \frac{1}{(2\pi)^m} \sum_{j;\, {\rm supp}\, \chi\, \cap\, {\rm supp}\, \chi_j \neq \emptyset} {\cal F}\chi * {\cal F}(\chi_j a),
\endeq
where $*$ is the convolution star. Using that
\begeq
\label{Wien13.3}
\abs{{\cal F}\chi(\xi)} \leq {\cal O}(1)\, \exp\left(-\frac{1}{C} \abs{\xi}^{1/s}\right),\quad \xi \in \real^m,
\endeq
and (\ref{Wien4}), we obtain that
\begin{multline}
\label{Wien13.4}
\abs{({\cal F}\chi * {\cal F}(\chi_j a))(\xi)} \leq \int \abs{{\cal F}(\chi)(\xi-\eta)} \abs{{\cal F}(\chi_j a)(\eta)}\, d\eta \\
\leq {\cal O}(1) \int \exp\left(-\frac{1}{C}\left(\abs{\xi-\eta}^{1/s} + \abs{\eta}^{1/s}\right)\right)\, d\eta \leq {\cal O}(1)\, \exp\left(-\frac{1}{{\cal O}(1)} \abs{\xi}^{1/s}\right).
\end{multline}
Here the last inequality in (\ref{Wien13.4}) follows by estimating separately the contributions coming from the regions of integration $\abs{\xi -\eta} \geq \abs{\xi}/2$ and $\abs{\xi -\eta} \leq \abs{\xi}/2$. Using (\ref{Wien13.2}) and (\ref{Wien13.4}), we get
\begeq
\label{Wien13.5}
\abs{{\cal F}\chi(\xi)} \leq {\cal O}(1)\, \exp\left(-\frac{1}{{\cal O}(1)} \abs{\xi}^{1/s}\right),\quad \xi \in \real^m.
\endeq
In this estimate we can replace $\chi$ by any translate of $\chi$, and by Fourier's inversion formula we can conclude therefore that $a\in {\cal G}^s_b(\real^m)$.

\bigskip
\noindent
{\it Remark}. Let $\chi_0 \in {\cal G}^s_0(\real^m)$ be real valued such that $\norm{\chi_0}_{L^2} = 1$, and let us set $\chi_t(x) = \chi_0(x-t)$, $t\in \real^m$. We then have the following natural analog of Proposition \ref{Wiener-Gevrey}: $a\in {\cal G}^s_b(\real^m)$, for some $s>1$, precisely when we have
\begeq
\label{Wien14}
\underset{t \in {\bf R}^m} {\rm sup} \abs{{\cal F}(\chi_t a)(\xi)} \leq {\cal O}(1)\, \exp\left(-\frac{1}{C}\abs{\xi}^{1/s}\right), \quad \xi \in \real^m,
\endeq
for some $C>0$. Indeed, this follows by arguing as in the proof of Proposition \ref{Wiener-Gevrey}, replacing (\ref{Wien8}) by the following consequence of the Fourier inversion formula,
\begeq
\label{Wien15}
a(x) = \frac{1}{(2\pi)^m} \int\!\!\!\int e^{ix\cdot \xi} \chi_t(x) {\cal F}(\chi_t a)(\xi)\, d\xi\,dt.
\endeq
Here we may also notice that the same characterization of the class ${\cal G}^s_b(\real^m)$ remains valid when $\chi_0(x) = 2^{m/4}\pi^{-m/4} e^{-\abs{x}^2}$ is the $L^2$--normalized real Gaussian. Indeed, the derivatives of $\chi_0$ obey the pointwise bounds
\begeq
\label{Wien16}
\abs{\partial^{\alpha} \chi_0(x)} \leq C^{1+\abs{\alpha}} (\alpha!)^{1/2} e^{-\abs{x}^2/C}, \quad x\in \real^m,\,\,\alpha \in \nat^m,
\endeq
for some $C\geq 1$, which is sufficient for the arguments to go through. Let us also remark that such derivative bounds are well known~\cite{LMPSXu}, and can also be obtained directly by means of the Cauchy inequalities.


\bigskip
\noindent
Let us proceed to make some additional remarks in the real setting, in preparation for the discussion on the FBI-Bargmann transform side. Let us set
\begeq
\label{Wien16.1}
e_0(x) = Ch^{-n/4} e^{-x^2/2h},\quad x\in \real^n,
\endeq
where $C>0$ is chosen so that $\norm{e_0}_{L^2} = 1$. The distribution kernel of the orthogonal projection $L^2(\real^n) \ni u\mapsto (u,e_0)_{L^2}\, e_0$ onto $\comp e_0$ is given by $K(x,y) = e_0(x) \overline{e_0(y)}$, and the semiclassical Weyl symbol of the orthogonal projection has the form
\begeq
\label{Wien16.2}
\int e^{-iy\cdot \xi/h} K\left(x + \frac{y}{2},x - \frac{y}{2}\right)\, dy.
\endeq
A simple computation shows that the integral in (\ref{Wien16.2}) is given by $\varphi_0((x,\xi)/h^{1/2})$, where
\begeq
\label{Wien16.3}
\varphi_0(x,\xi) = (4\pi)^{n/2} C^2 e^{-(x^2 + \xi^2)}.
\endeq
See also~\cite[Section 3]{Sj_XEDP_95}.

\bigskip
\noindent
We shall now pass to work on $\Lambda_{\Phi_0}$, and to this end let $\phi(x,y)$ be a holomorphic quadratic form on $\comp^n_x \times \comp^n_y$ such that
\begeq
\label{Wien16.4}
{\rm Im}\, \phi''_{yy}>0,\quad {\rm det}\, \phi''_{xy} \neq 0,
\endeq
and with the property that the associated complex linear canonical transformation
\begeq
\label{Wien16.5}
\kappa_{\phi}: \comp^{2n} \ni (y,-\phi'_y(x,y)) \mapsto (x,\phi'_x(x,y))\in \comp^{2n}
\endeq
satisfies
\begeq
\label{Wien16.6}
\kappa_{\phi}\left(\real^m\right) = \Lambda_{\Phi_0}.
\endeq
Associated to the quadratic form $\phi$ is the generalized Bargmann transformation
\begeq
\label{Wien16.7}
{\cal T}u(x,h) = C h^{-3n/4} \int e^{i\phi(x,y)/h}\, u(y)\, dy,
\endeq
where $C>0$ is chosen suitably so that the map ${\cal T}$ is unitary,
\begeq
\label{Wien16.8}
{\cal T}: L^2(\real^n) \rightarrow H_{\Phi_0}(\comp^n),
\endeq
see~\cite{Sj95},~\cite{HiSj15},~\cite{M_book}.

\bigskip
\noindent
Let $a\in {\cal G}^s_b(\Lambda_{\Phi_0})$ be uniformly Gevrey, for some $s>1$, and let us set $\chi_0 = \varphi_0 \circ \kappa_{\phi}^{-1}\in {\cal S}(\Lambda_{\Phi_0})$, where $\varphi_0$ is given in (\ref{Wien16.3}). Put also $\chi_T(X) = \chi_0(X-T)$, $T\in \Lambda_{\Phi_0}$. The discussion above, see (\ref{Wien14}), shows that
\begeq
\label{Wien17}
\abs{{\cal F}(\chi_T a)(Y)} \leq {\cal O}(1)\, \exp\left(-\frac{1}{C_0} \abs{Y}^{1/s}\right),\quad Y \in \Lambda_{\Phi_0},
\endeq
for some $C_0>0$, uniformly in $T\in \Lambda_{\Phi_0}$. Here ${\cal F}$ is the symplectic Fourier transformation on $\Lambda_{\Phi_0}$, introduced in (\ref{fourier1}). Letting ${\cal F}_h$ be the semiclassical symplectic Fourier transformation on $\Lambda_{\Phi_0}$, given in (\ref{comp1}), we can write, in view of the Fourier inversion formula,
\begeq
\label{Wien17.1}
\chi_T(X) a(X) = \frac{1}{(\pi h)^n} \int_{\Lambda_{\Phi_0}} e^{2i\sigma(X,Y)/h} {\cal F}_h(\chi_T a)(Y)\, dY.
\endeq
Multiplying (\ref{Wien17.1}) by $\chi_0((X-T)/h^{1/2}) = \varphi_0 \left(\kappa_{\phi}^{-1}(X-T)/h^{1/2}\right)$ and integrating with respect to $T\in \Lambda_{\Phi_0}$, we get
\begeq
\label{Wien18}
M(h) a(X) = \frac{1}{(\pi h)^n} \int\!\!\!\int_{\Lambda_{\Phi_0} \times \Lambda_{\Phi_0}} e^{2i\sigma(X,Y)/h} \chi_0\left(\frac{X-T}{h^{1/2}}\right) {\cal F}_h(\chi_T a)(Y)\, dY\, dT,
\endeq
where
\begeq
\label{Wien18.1}
M(h) = \int_{\Lambda_{\Phi_0}} \chi_0(T) \chi_0\left(\frac{T}{h^{1/2}}\right)\, dT \asymp h^n.
\endeq
We would like to pass to the Weyl quantizations in (\ref{Wien18}), and to this end we shall first take a closer look at the Weyl quantization of the Schwartz function
$$
\Lambda_{\Phi_0} \ni X\mapsto b(X) = b_{Y,T}(X) = e^{2i\sigma(X,Y)/h} \chi_0\left(\frac{X-T}{h^{1/2}}\right),\quad Y,T\in \Lambda_{\Phi_0}.
$$
Let us start with some preliminary observations and computations, closely related to~\cite[Section 3]{Sj08}. Let $\ell$ be a complex linear form on $\comp^{2n}$ such that $\ell|_{\Lambda_{\Phi_0}}$ is real, so that $H_{\ell} \in \Lambda_{\Phi_0}$, and let $a\in {\cal S}(\Lambda_{\Phi_0})$. A straightforward computation using (\ref{comp15}) shows that
\begeq
\label{Wien19}
\left(e^{i\ell/h}\#a\right)(X) = e^{i\ell(X)/h} a\left(X + \frac{H_{\ell}}{2}\right),
\endeq
and similarly we find
\begeq
\label{Wien20}
\left(a\# e^{i\ell/h}\right)(X) = e^{i\ell(X)/h} a\left(X - \frac{H_{\ell}}{2}\right).
\endeq
It follows from (\ref{Wien19}), (\ref{Wien20}) that
\begeq
\label{Wien21}
\left(e^{i\ell/h}\# a \# e^{i\ell/h}\right)(X) = e^{2i\ell(X)/h} a(X),
\endeq
\begeq
\label{Wien22}
\left(e^{i\ell/h}\# a \# e^{-i\ell/h}\right)(X) = a\left(X + H_{\ell}\right).
\endeq
Following (\ref{fourier4}), let us write $\ell_Y(X) = \sigma(X,Y)$, for $Y\in \Lambda_{\Phi_0}$, and notice that $H_{\ell_T} = T$. Using (\ref{Wien21}), (\ref{Wien22}) we get
\begin{multline}
\label{Wien23}
b(X) = e^{2i\sigma(X,Y)/h} \chi_0\left(\frac{X-T}{h^{1/2}}\right) \\
= \left(e^{i\ell_Y/h}\# e^{-i\ell_T/h}\# \chi_0\left(\frac{\cdot}{h^{1/2}}\right) \# e^{i\ell_T/h}\# e^{i\ell_Y/h}\right)(X),
\end{multline}
and therefore
\begin{multline}
\label{Wien24}
b^w(x,hD_x) \\
= e^{i\ell_Y(x,hD_x)/h}\circ e^{-i\ell_T(x,hD_x)/h} \circ \chi_0^w\left(\frac{(x,hD_x)}{h^{1/2}}\right) \circ e^{i\ell_T(x,hD_x)/h} \circ e^{i\ell_Y(x,hD_x)/h}.
\end{multline}
Using (\ref{comp4}) we infer that
\begeq
\label{Wien24.1}
e^{i\ell_Y(x,hD_x)/h}\circ e^{-i\ell_T(x,hD_x)/h} = e^{i\ell_{Y-T}(x,hD_x)/h} e^{i\sigma(Y,-T)/2h} = e^{i\ell_{Y-T}(x,hD_x)/h} e^{i\sigma(T,Y)/2h},
\endeq
\begeq
\label{Wien24.2}
e^{i\ell_T(x,hD_x)/h} \circ e^{i\ell_Y(x,hD_x)/h} = e^{i\ell_{T+Y}(x,hD_x)/h} e^{i\sigma(T,Y)/2h},
\endeq
and combining (\ref{Wien24}), (\ref{Wien24.1}), and (\ref{Wien24.2}) we get
\begeq
\label{Wien25}
b^w(x,hD_x) = e^{i\sigma(T,Y)/h} e^{i\sigma((x,hD_x),Y-T)/h} \circ \chi_0^w\left(\frac{(x,hD_x)}{h^{1/2}}\right) \circ e^{i\sigma((x,hD_x),Y+T)/h}.
\endeq
An application of the exact Egorov theorem~\cite{Sj95},~\cite{HiSj15}, gives that
\begeq
\label{Wien25.1}
\chi_0^w\left(\frac{(x,hD_x)}{h^{1/2}}\right) = {\cal T} \circ \varphi_0^w\left(\frac{(x,hD_x)}{h^{1/2}}\right) \circ {\cal T}^{-1},
\endeq
and we conclude therefore that the operator in (\ref{Wien25.1}) is a rank one orthogonal projection on $H_{\Phi_0}(\comp^n)$, given by
\begeq
\label{Wien25.2}
\chi_0^w\left(\frac{(x,hD_x)}{h^{1/2}}\right) u = (u,v_0)\, v_0,\quad u\in H_{\Phi_0}(\comp^n).
\endeq
Here $v_0 = {\cal T} e_0$ and $(\cdot, \cdot)$ stands for the scalar product in $H_{\Phi_0}(\comp^n)$. Let us mention explicitly that we owe the idea of using Gaussians to pass to rank one projections to~\cite{GrHe}.

\medskip
\noindent
For future reference, let us also notice that an application of the exact (quadratic) stationary phase together with (\ref{Wien16.1}), (\ref{Wien16.7}) allows us to conclude that
\begeq
\label{Wien25.3}
v_0(x) = ({\cal T} e_0)(x,h) = C h^{-n/2} e^{ig(x)/h}, \quad C\neq 0,
\endeq
where $g$ is a holomorphic quadratic form on $\comp^n$. The strict positivity of the complex Lagrangian plane $\eta = iy$, $y\in \comp^n$, associated to the state $e_0$ in (\ref{Wien16.1}) implies that
\begeq
\label{Wien25.4}
\Phi_0(x) + {\rm Im}\, g(x) \asymp \abs{x}^2, \quad x\in \comp^n,
\endeq
see~\cite[Theorem 2.1]{CoHiSj}.

\bigskip
\noindent
Using (\ref{Wien25}) and (\ref{Wien25.2}), we get, using the unitarity of magnetic trans\-la\-tions on $H_{\Phi_0}(\comp^n)$,
\begin{multline}
\label{Wien25.5}
b^w(x,hD_x)u = e^{i\sigma(T,Y)/h} e^{i\sigma((x,hD_x),Y-T)/h} \circ \chi_0^w\left(\frac{(x,hD_x)}{h^{1/2}}\right) \circ e^{i\sigma((x,hD_x),Y+T)/h} u \\
= e^{i\sigma(T,Y)/h} (u, e^{-i\sigma((x,hD_x),Y+T)/h}v_0)\, e^{i\sigma((x,hD_x),Y-T)/h}v_0.
\end{multline}
Passing to the Weyl quantizations in (\ref{Wien18}), we obtain therefore, in view of (\ref{Wien25.5}),
\begin{multline}
\label{Wien26}
M(h) a^w(x,hD_x)u \\
= \frac{1}{(\pi h)^n} \int\!\!\!\int_{\Lambda_{\Phi_0} \times \Lambda_{\Phi_0}} e^{\frac{i\sigma(T,Y)}{h}} {\cal F}_h(\chi_T a)(Y)
(u, e^{-\frac{i\sigma((x,hD_x),Y+T)}{h}}v_0)\, e^{\frac{i\sigma((x,hD_x),Y-T)}{h}}v_0\, dY\, dT.
\end{multline}
Making the change of variables $Y' = Y -T$, $T' = -Y -T$, and using that $2\sigma(T,Y) = \sigma(Y',T')$, we obtain after dropping the primes,
\begin{multline}
\label{Wien26.1}
M(h) a^w(x,hD_x)u \\
= \frac{C}{h^n} \int\!\!\!\int_{(\Lambda_{\Phi_0})^2} e^{\frac{i\sigma(Y,T)}{2h}}
{\cal F}_h(\chi_{-\frac{Y+T}{2}} a)\left(\frac{Y-T}{2}\right)
(u, e^{\frac{i\sigma((x,hD_x),T)}{h}}v_0)\, e^{\frac{i\sigma((x,hD_x),Y)}{h}} v_0\, dY\, dT.
\end{multline}
Here we have incorporated the non-vanishing constant Jacobian into the (new) constant $C\neq 0$. We have therefore represented the operator $a^w(x,hD_x)$ as a super\-po\-siti\-on of rank one ker\-nels. The decomposition (\ref{Wien26.1}) can be regarded as the Bargmann transform side analogue of the corresponding decomposition in the real setting, established in~\cite{GrHe}.

\bigskip
\noindent
Let us next record the following observation, closely related to the computations in Section \ref{magn_transl}.
\begin{lemma}
\label{magn_transl_expl}
Let $\ell(x,\xi)$ be a complex linear form on $\comp^{2n}$ such that the restriction $\ell|_{\Lambda_{\Phi_0}}$ is real, and let us represent $\ell$ in the form {\rm (\ref{eq2.3})},
$$
\ell(x,\xi) = -\frac{2}{i}\frac{\partial \Phi_0}{\partial x}(x^*)\cdot x + x^*\cdot \xi,
$$
for some unique $x^* \in \comp^n$. There exists $C_{x^*,h}\in \comp$ with $\abs{C_{x^*,h}} = 1$, such that
\begeq
\label{Wien26.2}
e^{i\ell(x,hD_x)/h} u(x) = C_{x^*,h} e^{-2i{\rm Im}\,(\Phi''_{0,xx}x^* \cdot x)/h} e^{i\sigma(X,X^*)/2h} e^{(\Phi_0(x) - \Phi_0(x + x^*))/h} u(x + x^*).
\endeq
Here $u\in H_{\Phi_0}(\comp^n)$, and $X\in \Lambda_{\Phi_0}$, $X^* = H_{\ell} \in \Lambda_{\Phi_0}$ are the points in $\Lambda_{\Phi_0}$ above $x$, $x^* \in \comp^n$, respectively.
\end{lemma}
\begin{proof}
This result follows by a direct computation, using (\ref{eq2.2}), (\ref{eq2.5}), (\ref{eq2.6}), as well as the following general expression for the complex symplectic (2,0)--form $\sigma$ on $\comp^{2n}$, restricted to $\Lambda_{\Phi_0}$,
\begeq
\label{Wien26.3}
\sigma(X,X^*) = -4 {\rm Im}\, \left(\Phi''_{0,\overline{x}x}x\cdot \overline{x^*}\right), \quad X,X^*\in \Lambda_{\Phi_0}.
\endeq
\end{proof}

\bigskip
\noindent
{\it Remark}. Let $u,v\in H_{\Phi_0}(\comp^n)$. It follows from Lemma \ref{magn_transl_expl} that the scalar product $(e^{i\ell(x,hD_x)/h} u, v)_{H_{\Phi_0}}$, viewed as a function of $x^*\in \comp^n$, or equivalently as a function of $X^*\in \Lambda_{\Phi_0}$, can be regarded as the twisted convolution of the functions $e^{-\Phi_0/h}u$, $e^{-\Phi_0/h}v\in L^2(\comp^n)$ in the sense of (\ref{comp7}), after these have been modified by some unimodular factors. It follows, in particular, that the function
$$
\comp^n \ni x^* \mapsto (e^{i\ell(x,hD_x)/h} u, v)_{H_{\Phi_0}}\in L^2(\comp^n),
$$
see~\cite{T02}. Ignoring the unimodular factors in (\ref{Wien26.2}), we get the more elementary pointwise estimate,
\begeq
\label{Wien27}
\abs{(e^{i\ell(x,hD_x)/h} u, v)_{H_{\Phi_0}}(x^*)} \leq \int_{{\bf C}^n} e^{-\Phi_0(x + x^*)/h} \abs{u(x+x^*)}\, e^{-\Phi_0(x)/h}\abs{v(x)}\, L(dx),
\endeq
which will be sufficient in what follows.

\bigskip
\noindent
We now come to complete the proof of Theorem \ref{theo_main}. When doing so, let us write, using (\ref{Wien26.1}), (\ref{Wien26.2}), and (\ref{Wien27}),
\begin{multline}
\label{Wien37}
M(h) \abs{a^w(x,hD_x)u(x)} e^{-\Phi_0(x)/h} \\
\leq \frac{{\cal O}(1)}{h^{2n}} \int\!\!\!\!\int\!\!\!\!\int_{({\bf C}^n)^3} {\cal U}\left(\frac{y-t}{h}\right) e^{-\Phi_0(z+t)/h}\abs{v_0(z+t)}
e^{-\Phi_0(x+y)/h}\abs{v_0(x+y)} \\
\abs{u(z)} e^{-\Phi_0(z)/h}\, L(dy)\, L(dt)\, L(dz),
\end{multline}
where, in view of (\ref{Wien17}),
\begeq
\label{Wien38}
{\cal U}(y) = \exp\left(-\frac{1}{C_0} \abs{y}^{1/s}\right),\quad y\in \comp^n.
\endeq
Letting $\Phi_1 = \Phi_0 + f\in C^{1,1}(\comp^n;\real)$ be such that (\ref{fourier5}) holds, we obtain next, making use of (\ref{Wien18.1}), (\ref{Wien25.3}), (\ref{Wien25.4}), and (\ref{Wien37}),
\begeq
\label{Wien39}
\abs{a^w(x,hD_x)u(x)} e^{-\Phi_1(x)/h} \leq \int_{{\bf C}^n} K(x,z) \abs{u(z)} e^{-\Phi_1(z)/h}\, L(dz),
\endeq
where
\begin{multline}
\label{Wien40}
K(x,z) \leq \\
\frac{{\cal O}(1)}{h^{4n}} \int\!\!\!\int_{{\bf C}^n \times {\bf C}^n} {\cal U}\left(\frac{y-t}{h}\right) e^{-\abs{z+t}^2/Ch} e^{-\abs{x+y}^2/Ch}\, e^{(f(z)-f(x))/h}\, L(dt)\, L(dy).
\end{multline}
We would like to show that the kernel $K(x,z)$ is dominated pointwise by an $L^1$ convolution kernel, in order to be able to apply Schur's lemma to (\ref{Wien39}). To this end let us consider the $t$--integration in (\ref{Wien40}) first, estimating the integral
\begeq
\label{Wien41}
\frac{1}{h^{2n}}\int_{{\bf C}^n} {\cal U}\left(\frac{y-t}{h}\right) e^{-\abs{z+t}^2/Ch}\, L(dt) =
\frac{1}{h^{2n}}\int_{{\bf C}^n} {\cal U}\left(\frac{t}{h}\right) e^{-\abs{z+y-t}^2/Ch}\, L(dt) = I_1 + I_2.
\endeq
Here
\begeq
\label{Wien42}
I_1 = \frac{1}{h^{2n}}\int_{\abs{z+y-t} \geq \abs{z+y}/2} {\cal U}\left(\frac{t}{h}\right) e^{-\abs{z+y-t}^2/Ch}\, L(dt) \leq
\norm{{\cal U}}_{L^1} e^{-\abs{z+y}^2/4Ch},
\endeq
and, in view of (\ref{Wien38}), we have
\begin{multline}
\label{Wien43}
I_2 = \frac{1}{h^{2n}}\int_{\abs{z+y-t} \leq \abs{z+y}/2} {\cal U}\left(\frac{t}{h}\right) e^{-\abs{z+y-t}^2/Ch}\, L(dt) \\
\leq \frac{1}{h^{2n}} \int_{\abs{z+y-t} \leq \abs{z+y}/2} {\cal U}\left(\frac{t}{h}\right)\, L(dt) \leq {\cal O}(1)\, \exp\left(-\frac{1}{C_0}
\left(\frac{\abs{z+y}}{h}\right)^{1/s}\right).
\end{multline}
Here we have also used that $\abs{z+y}\leq 2\abs{t}$ in the region of integration in (\ref{Wien43}).

\medskip
\noindent
Combining (\ref{Wien40}), (\ref{Wien41}), (\ref{Wien42}), and (\ref{Wien43}), we see that
\begeq
\label{Wien44}
K(x,z) \leq K_1(x,z) + K_2(x,z),
\endeq
where
\begin{multline}
\label{Wien45}
K_1(x,z) \leq e^{(f(z)-f(x))/h}\, \frac{{\cal O}(1)}{h^{2n}} \int_{{\bf C}^n} e^{-\abs{z+y}^2/Ch}\, e^{-\abs{x+y}^2/Ch}\, L(dy) \\
= e^{(f(z)-f(x))/h}\, \frac{{\cal O}(1)}{h^{2n}} \int_{{\bf C}^n} e^{-\abs{y}^2/Ch}\, e^{-\abs{z-x+y}^2/Ch}\, L(dy),
\end{multline}
and
\begin{multline}
\label{Wien46}
K_2(x,z) \leq e^{(f(z)-f(x))/h}\, \frac{{\cal O}(1)}{h^{2n}} \int_{{\bf C}^n} \exp\left(-\frac{1}{C_0}
\left(\frac{\abs{z+y}}{h}\right)^{1/s}\right) e^{-\abs{x+y}^2/Ch}\, L(dy) \\
= e^{(f(z)-f(x))/h}\, \frac{{\cal O}(1)}{h^{2n}} \int_{{\bf C}^n} \exp\left(-\frac{1}{C_0}
\left(\frac{\abs{z-x+y}}{h}\right)^{1/s}\right) e^{-\abs{y}^2/Ch}\, L(dy).
\end{multline}
When estimating the contribution $K_1(x,z)$ in (\ref{Wien45}), we notice that considering separately the regions of integration $\abs{z-x+y}\leq \abs{z-x}/2$ and
$\abs{z-x+y}\geq \abs{z-x}/2$, and using that
$$
\frac{1}{h^n} \int_{{\bf C}^n} e^{-\abs{y}^2/Ch}\, L(dy) = {\cal O}(1),
$$
we get
\begeq
\label{Wien47}
\frac{{\cal O}(1)}{h^{n}} \int_{{\bf C}^n} e^{-\abs{y}^2/Ch}\, e^{-\abs{z-x+y}^2/Ch}\, L(dy) \leq {\cal O}(1)\, e^{-\abs{x-z}^2/Ch},
\endeq
and therefore,
\begeq
\label{Wien48}
K_1(x,z) \leq \frac{{\cal O}(1)}{h^n} e^{(f(z)-f(x))/h} e^{-\abs{x-z}^2/Ch}.
\endeq
Alternatively, the estimate (\ref{Wien47}) can be obtained by an application of the exact stationary phase to the integral in the left hand side of (\ref{Wien47}). Arguing similarly, we find that
\begin{multline}
\label{Wien49}
\frac{{\cal O}(1)}{h^{n}} \int_{{\bf C}^n} \exp\left(-\frac{1}{C_0}
\left(\frac{\abs{z-x+y}}{h}\right)^{1/s}\right) e^{-\abs{y}^2/Ch}\, L(dy) \\
\leq {\cal O}(1)\, \exp\left(-\frac{1}{C_0} \left(\frac{\abs{z-x}}{h}\right)^{1/s}\right) + {\cal O}(1)\, e^{-\abs{x-z}^2/Ch}.
\end{multline}
Combining (\ref{Wien44}), (\ref{Wien48}), (\ref{Wien46}), and (\ref{Wien49}),  we get
\begeq
\label{Wien50}
K(x,z) \leq \frac{{\cal O}(1)}{h^n} e^{(f(z)-f(x))/h} e^{-\abs{x-z}^2/Ch} + \frac{{\cal O}(1)}{h^n} e^{(f(z)-f(x))/h} \exp\left(-\frac{1}{C_0} \left(\frac{\abs{z-x}}{h}\right)^{1/s}\right).
\endeq
To handle the second term in the right hand side of (\ref{Wien50}) we write, following (\ref{fourier8}) and using (\ref{fourier5}),
\begeq
\label{Wien51}
f(z) - f(x) \leq \frac{1}{{\cal O}(1)} h^{1 - \frac{1}{s}} {\rm min}\, \left(1,\abs{z-x}\right) \leq \frac{1}{{\cal O}(1)} h^{1 - \frac{1}{s}}\abs{z-x}^{1/s}.
\endeq
The Schur norm of the second term in the right hand side of (\ref{Wien50}) is therefore ${\cal O}(1)$, provided that the implicit constant in (\ref{Wien51}) is large enough, and we only need to estimate the Schur norm of the first term in the right hand side of (\ref{Wien50}). Using (\ref{Wien51}), we see that it suffices to control the $L^1$--norm
\begeq
\label{Wien52}
\frac{1}{h^n} \int_{{\bf C}^n} e^{-\abs{x}^2/Ch}\, \exp\left(\frac{h^{1-\frac{1}{s}}\abs{x}}{{\cal O}(1)\, h}\right)\, L(dx) = I_1 + I_2,
\endeq
where
\begeq
\label{Wien53}
I_1  = \frac{1}{h^n} \int_{\abs{x} \geq \widetilde{C} h^{1-\frac{1}{s}}} e^{-\abs{x}^2/Ch} \exp\left(\frac{h^{1-\frac{1}{s}}\abs{x}}{{\cal O}(1)\, h}\right)\, L(dx),
\endeq
and
\begeq
\label{Wien54}
I_2  = \frac{1}{h^n} \int_{\abs{x} \leq \widetilde{C} h^{1-\frac{1}{s}}} e^{-\abs{x}^2/Ch}\, \exp\left(\frac{h^{1-\frac{1}{s}}\abs{x}}{{\cal O}(1)\, h}\right)\, L(dx).
\endeq
Taking the constant $\widetilde{C}>0$ sufficiently large, we get
\begeq
\label{Wien55}
I_1 \leq \frac{1}{h^n} \int \exp\left(-\frac{\abs{x}^2}{{\cal O}(1)h}\right)\, L(dx) = {\cal O}(1).
\endeq
Furthermore,
\begin{multline}
\label{Wien56}
I_2 \leq \left(\frac{1}{h^n} \int_{\abs{x} \leq \widetilde{C} h^{1-\frac{1}{s}}} e^{-\abs{x}^2/Ch}\,L(dx)\right)
\exp\left(\frac{\widetilde{C} h^{2-\frac{2}{s}}}{{\cal O}(1) h}\right) \\
\leq {\cal O}(1)\, \exp\left({\cal O}(1) h^{1-\frac{2}{s}}\right).
\end{multline}
Recalling that $s\geq 2$, we conclude, in view of (\ref{Wien50}), (\ref{Wien51}), (\ref{Wien52}), (\ref{Wien55}), and (\ref{Wien56}), that the Schur norm of the kernel $K(x,z)$ is ${\cal O}(1)$. Applying Schur's lemma to (\ref{Wien39}), we get therefore,
$$
{\rm Op}_h^w(a) = {\cal O}(1): H_{\Phi_1}(\comp^n) \rightarrow H_{\Phi_1}(\comp^n).
$$
The proof of Theorem \ref{theo_main} is complete.

\bigskip
\noindent
{\it Remark}. The purpose of this remark is to verify that the decomposition (\ref{Wien26.1}) can also be used to give a direct proof of the $L^2$--boundedness result for the Wiener algebra of pseudodifferential operators, established in~\cite{Sj94},~\cite{Sj_XEDP_95}.
Indeed, the fact that decompositions such as (\ref{Wien26.1}) are useful to this end is well known in the real setting~\cite{GrHe}, and the observation here is that working on the FBI--Bargmann transform side seems to make the computations and estimates particularly natural. See also~\cite[Section 5]{Sj_XEDP_95}.

\medskip
\noindent
Let us therefore replace (\ref{Wien17}) by the weaker assumption,
\begeq
\label{Wien57}
\abs{{\cal F}(\chi_T a)(Y)} \leq U(Y),\quad Y \in \Lambda_{\Phi_0},
\endeq
uniformly in $T\in \Lambda_{\Phi_0}$. Here $U\in L^1(\Lambda_{\Phi_0})$. Setting
\begeq
\label{Wien58}
F(T) = (u, e^{\frac{i\sigma((x,hD_x),T)}{h}}v_0)_{H_{\Phi_0}},\quad T\in \Lambda_{\Phi_0}\simeq \comp^n,
\endeq
we get, in view of (\ref{Wien27}) and the Young inequality,
\begeq
\label{Wien59}
\norm{F}_{L^2(\Lambda_{\Phi_0})} \leq {\cal O}(1) \norm{u}_{H_{\Phi_0}} \norm{e^{-\Phi_0/h}v_0}_{L^1} \leq {\cal O}(h^{n/2}) \norm{u}_{H_{\Phi_0}}.
\endeq
Here we have also used (\ref{Wien25.3}), (\ref{Wien25.4}). An application of Schur's lemma together with (\ref{Wien57}) and (\ref{Wien59}) allows us next to conclude that the function
\begeq
\label{Wien60}
G(Y) := \frac{C}{h^n} \int_{\Lambda_{\Phi_0}} e^{\frac{i\sigma(Y,T)}{2h}} {\cal F}_h(\chi_{-(Y+T)/2} a)\left(\frac{Y-T}{2}\right) F(T)\, dT \in L^2(\Lambda_{\Phi_0})
\endeq
and we have
\begeq
\label{Wien61}
\norm{G}_{L^2(\Lambda_{\Phi_0})} \leq {\cal O}(1)\norm{F}_{L^2(\Lambda_{\Phi_0})} \leq {\cal O}(h^{n/2}) \norm{u}_{H_{\Phi_0}}.
\endeq
Using (\ref{Wien26.1}), (\ref{Wien58}), and (\ref{Wien60}), we can write
\begeq
\label{Wien62}
M(h) a^w(x,hD_x)u(x) = \int_{\Lambda_{\Phi_0}} G(Y) e^{\frac{i\sigma((x,hD_x),Y)}{h}} v_0(x)\, dY,
\endeq
and an application of Lemma \ref{magn_transl_expl} gives the pointwise estimate,
\begeq
\label{Wien63}
M(h) \abs{a^w(x,hD_x)u(x)}e^{-\Phi_0(x)/h} \leq \int_{\Lambda_{\Phi_0}} \abs{G(Y)} e^{-\Phi_0(x+y)/h} \abs{v_0(x+y)}\, dY.
\endeq
Here we have written $Y = (y,\eta)\in \Lambda_{\Phi_0}$. Applying the Young inequality once more we get, using also (\ref{Wien61}),
\begeq
\label{Wien64}
M(h) \norm{a^w u}_{H_{\Phi_0}} \leq \norm{G}_{L^2} \norm{e^{-\Phi_0/h}v_0}_{L^1}
\leq {\cal O}(h^{n/2}) \norm{G}_{L^2} \leq {\cal O}(h^n) \norm{u}_{H_{\Phi_0}}.
\endeq
Recalling finally (\ref{Wien18.1}) we conclude that
\begeq
\label{Wien65}
\norm{a^w u}_{H_{\Phi_0}} \leq {\cal O}(1) \norm{u}_{H_{\Phi_0}}, \quad u\in H_{\Phi_0}(\comp^n),
\endeq
provided that (\ref{Wien57}) holds. We have therefore recovered the $L^2$--boundedness result of~\cite{Sj94},~\cite{Sj_XEDP_95} in the $H_{\Phi_0}$--setting.

\begin{appendix}
\section{Weyl composition of symbols}
\label{comp_Gevrey}
\setcounter{equation}{0}
Let $(W,\sigma)$ be a real symplectic vector space of dimension $2n$ and let
\begeq
\label{comp1}
{\cal F}_h u(X) = \frac{1}{h^n} {\cal F}u\left(\frac{X}{h}\right) = \frac{1}{(\pi h)^n} \int e^{2i\sigma(X,Y)/h} u(Y)\, dY, \quad u\in {\cal S}(W),\quad 0 < h \leq 1,
\endeq
be the semiclassical (twisted) Fourier transformation on $W$. Here the map ${\cal F}$ is given in (\ref{fourier1}). We have ${\cal F}_h^2 = I$ on ${\cal S}'(W)$.

\bigskip
\noindent
We shall carry out a familiar computation composing two semiclassical Weyl quantizations, see~\cite[Chapter 7]{DiSj} for such computations on the real side. Let $a,b\in {\cal S}'(\Lambda_{\Phi_0})$ be such that ${\cal F}a$, ${\cal F}b\in L^1(\Lambda_{\Phi_0})$ and let us write following (\ref{fourier3}),
\begeq
\label{comp2}
a^w(x,hD_x) = \frac{1}{(\pi h)^{n}} \int_{\Lambda_{\Phi_0}} e^{2i\sigma((x,hD_x),Y)/h} {\cal F}_h a(Y) \,dY,
\endeq
\begeq
\label{comp3}
b^w(x,hD_x) = \frac{1}{(\pi h)^{n}} \int_{\Lambda_{\Phi_0}} e^{2i\sigma((x,hD_x),Y)/h} {\cal F}_h b(Y) \,dY.
\endeq
Using the composition law for magnetic translations
\begeq
\label{comp4}
e^{2i\sigma((x,hD_x),Y)/h} e^{2i\sigma((x,hD_x),Z)/h} = e^{2i\sigma((x,hD_x),Y+Z)/h} e^{2i\sigma(Y,Z)/h},
\endeq
see~\cite[(7.11)]{DiSj} for the corresponding result in the real domain, and making the change of variables $(Y,Z)\mapsto (Y+Z,Z)$, we get that the composition $a^w(x,hD_x)\circ b^w(x,hD_x)$ of the operators in (\ref{comp2}), (\ref{comp3}), is given by
\begin{multline}
\label{comp5}
\frac{1}{(\pi h)^{2n}} \int\!\!\!\int_{\Lambda_{\Phi_0} \times \Lambda_{\Phi_0}} e^{2i\sigma((x,hD_x),Y+Z)/h} e^{2i\sigma(Y,Z)/h} {\cal F}_h a(Y)
{\cal F}_h b(Z)\, dY\, dZ \\
= \frac{1}{(\pi h)^n} \int_{\Lambda_{\Phi_0}} e^{2i\sigma((x,hD_x),Y)/h} {\cal F}_h c(Y)\, dY = c^w(x,hD_x).
\end{multline}
Here
\begeq
\label{comp6}
{\cal F}_h c(X) = \frac{1}{(\pi h)^n} \int_{\Lambda_{\Phi_0}} e^{2i\sigma(X,Z)/h}
{\cal F}_ha(X-Z) {\cal F}_h b(Z)\, dZ \in L^1(\Lambda_{\Phi_0}).
\endeq
We shall now compute the semiclassical Fourier transform of the expression in the right hand side of (\ref{comp6}), leading to an integral representation formula for the symbol $c = a\# b \in L^{\infty}(\Lambda_{\Phi_0})\cap C(\Lambda_{\Phi_0})$. To this end, following~\cite{T02}, it will be convenient to introduce the (non-commutative) twisted convolution product on $\Lambda_{\Phi_0}$,
\begeq
\label{comp7}
(u *_{\sigma} v)(X) = \int_{\Lambda_{\Phi_0}} e^{2i\sigma(X,Y)/h} u(X-Y) v(Y)\, dY,
\endeq
where $u,v\in L^1(\Lambda_{\Phi_0})$, so that
\begeq
\label{comp8}
{\cal F}_h (a\#b) = \frac{1}{(\pi h)^n} {\cal F}_h a *_{\sigma} {\cal F}_h b.
\endeq

\medskip
\noindent
We have the following result, due to~\cite{T02}, whose proof we give for the convenience of the reader only.
\begin{prop}
\label{tw_conv}
We have if $u,v\in L^1(\Lambda_{\Phi_0})$,
\begeq
\label{comp9}
{\cal F}_h(u*_{\sigma} v) = ({\cal F}_h u) *_{\sigma} v.
\endeq
\end{prop}
\begin{proof}
Using (\ref{comp1}), (\ref{comp7}), let us write
\begin{multline}
\label{comp10}
{\cal F}_h(u*_{\sigma} v)(X) = \frac{1}{(\pi h)^n} \int_{\Lambda_{\Phi_0}} e^{2i\sigma(X,Y)/h} (u*_{\sigma} v)(Y)\, dY \\
= \frac{1}{(\pi h)^n} \int\!\!\!\int_{\Lambda_{\Phi_0} \times \Lambda_{\Phi_0}} e^{2i\sigma(X,Y)/h} e^{2i\sigma(Y,Z)/h} u(Y-Z) v(Z)\, dY\,dZ.
\end{multline}
On the other hand, we compute
\begin{multline}
\label{comp11}
(({\cal F}_h u) *_{\sigma} v)(X) = \int_{\Lambda_{\Phi_0}} e^{2i\sigma(X,Z)/h} ({\cal F}_h u)(X-Z) v(Z)\, dZ \\
= \frac{1}{(\pi h)^n} \int\!\!\!\int_{\Lambda_{\Phi_0} \times \Lambda_{\Phi_0}} e^{2i\sigma(X,Z)/h} e^{2i\sigma(X-Z,Y)/h} u(Y) v(Z)\, dY\,dZ.
\end{multline}
Making the change of variables $(Y,Z) \mapsto (Y-Z,Z)$, we can rewrite (\ref{comp11}) as follows,
\begeq
\label{comp12}
(({\cal F}_h u) *_{\sigma} v)(X) = \frac{1}{(\pi h)^n} \int\!\!\!\int_{\Lambda_{\Phi_0} \times \Lambda_{\Phi_0}} e^{2i\sigma(X,Z)/h} e^{2i\sigma(X-Z,Y-Z)/h} u(Y-Z) v(Z)\, dY\,dZ.
\endeq
Here
$$
\sigma(X,Z) + \sigma(X-Z,Y-Z) = \sigma(X,Z) + \sigma(X,Y) - \sigma(X,Z) - \sigma(Z,Y) = \sigma(X,Y) + \sigma(Y,Z),
$$
and therefore the expressions (\ref{comp10}) and (\ref{comp12}) agree.
\end{proof}

\bigskip
\noindent
Combining Proposition \ref{tw_conv} and (\ref{comp8}) with the fact that ${\cal F}_h^2 = I$, we get
\begeq
\label{comp13}
c = a\#b = \frac{1}{(\pi h)^n} a *_{\sigma} {\cal F}_h b,
\endeq
and therefore, assuming for simplicity that $a,b\in {\cal S}(\Lambda_{\Phi_0})$, we get
\begin{multline}
\label{comp14}
c(X) = \frac{1}{(\pi h)^n} \int_{\Lambda_{\Phi_0}} e^{2i\sigma(X,Y)/h} a(X-Y) {\cal F}_h b(Y)\, dY \\
= \frac{1}{(\pi h)^{2n}} \int\!\!\!\int_{\Lambda_{\Phi_0} \times \Lambda_{\Phi_0}} e^{2i\sigma(X,Y)/h}e^{2i\sigma(Y,Z)/h} a(X-Y) b(Z)\, dY\,dZ \\
= \frac{1}{(\pi h)^{2n}} \int\!\!\!\int_{\Lambda_{\Phi_0} \times \Lambda_{\Phi_0}} e^{2i\sigma(X-Z,Y)/h} a(X-Y) b(Z)\, dY\,dZ.
\end{multline}
We obtain finally, after a change of variables,
\begeq
\label{comp15}
c(X) = (a\#b)(X) =  \frac{1}{(\pi h)^{2n}} \int\!\!\!\int_{\Lambda_{\Phi_0} \times \Lambda_{\Phi_0}} e^{-2i\sigma(Y,Z)/h} a(X+Y)b(X+Z)\, dY\,dZ.
\endeq

\medskip
\noindent
{\it Remark}. The integral representation formula (\ref{comp15}) can also be inferred from the corresponding expression for the Weyl symbol of the composition $a^w(x,hD_x) \circ b^w(x,hD_x)$ in the real domain~\cite[Chapter 4]{Zw_book}, thanks to the metaplectic invariance of the Weyl calculus~\cite{Sj95},~\cite{HiSj15}.

\end{appendix}

\end{document}